\documentclass[12pt, reqno]{amsart}
\usepackage[utf8]{inputenc}
\usepackage{amsmath,amssymb,amsthm, mathrsfs}
\usepackage{indentfirst}
\usepackage{comment,relsize,braket}
\usepackage{enumerate}
\usepackage{graphicx}
\usepackage{color}
\usepackage[all]{xy}
\usepackage{fullpage}
\usepackage{tikz}
\usepackage[colorlinks, linkcolor=blue!70!black]{hyperref}
\usepackage[top=1in, bottom=1.2in, left=0.9in, right=0.9in]{geometry}
\usepackage{times}
\usepackage{tfrupee}
\usetikzlibrary{chains,fit}
\usetikzlibrary{shapes,snakes}
\usetikzlibrary{graphs}
\usepackage{graphicx,adjustbox}
\usepackage{hyperref}
\usepackage{amscd}

\usepackage{float}
\usepackage{caption}
\usepackage{subcaption}

\usepackage{cleveref}

\usepackage[backend=bibtex]{biblatex}
\numberwithin{equation}{section}
\newtheorem{theorem}{Theorem}[section]
\newtheorem{proposition}[theorem]{Proposition}

\newtheorem{lemma}[theorem]{Lemma}

\newtheorem{definition}[theorem]{Definition}

\newtheorem{remark}[theorem]{Remark}

\newtheorem{setup}[theorem]{Setup}

\newcommand{\ov}{\overline}
\newcommand{\ini}{\mathrm{in}}
\newcommand{\reg}{\mathrm{reg}}

\title[Regularity of Cohen-Macaulay binomial edge ideals]{Regularity of two classes of Cohen-Macaulay binomial edge ideals}

\author{Om Prakash Bhardwaj}
\address{Department of Mathematics, Indian Institute of Technology Bombay, Powai, Maharashtra-400076, India}
\email{om.prakash@math.iitb.ac.in}
\thanks{}

\author{Kamalesh Saha}
\address{Chennai Mathematical Institute, Siruseri, Chennai, Tamil Nadu 603103, India}
\email{ksaha@cmi.ac.in; kamalesh.saha44@gmail.com}
\thanks{}


\begin{document}
\begin{abstract}
Some recent investigations indicate that for the classification of Cohen-Macaulay binomial edge ideals, it suffices to consider biconnected graphs with some whiskers attached (in short, `block with whiskers'). This paper provides explicit combinatorial formulae for the Castelnuovo-Mumford regularity of two specific classes of Cohen-Macaulay binomial edge ideals: (i) chain of cycles with whiskers and (ii) $r$-regular $r$-connected block with whiskers. For the first type, we introduce a new invariant of graphs in terms of the number of blocks in certain induced block graphs, and this invariant may help determine the regularity of other classes of binomial edge ideals. For the second type, we present the formula as a linear function of $r$.
\end{abstract}

\maketitle

% \section{Introduction}

% \section{Class-1}
% \begin{theorem}
%     Let $\lambda$ be the number of whiskers that are on different cycles and not on a common vertex of two cycles. Then $\mathrm{reg}(I_G) = 2 \cdot \#EC_4 + \# OC_4 + \# C_3 + \lambda.$
% Where $EC_4 = 4-$cycles with extreme whiskers attached, $OC_4$ = other 4-cycles.
% \end{theorem}

% Note that, $\lambda$ can take the values $0,1$ or $2$.

% \begin{figure}
%     \centering
%     \includegraphics[width=10cm]{class1-1.jpg}
% \end{figure}

% \begin{figure}
%     \centering
%     \includegraphics[width=10cm]{class1.jpg}
% \end{figure}

\section{Introduction}

In combinatorial commutative algebra, researchers often connect ideals in polynomial rings to various combinatorial objects, exploring their relevance and applications in other research areas. Generally, people investigate algebraic properties and invariants of these ideals in relation to the corresponding combinatorial objects. One notable class of ideals that has garnered significant interest in recent years is the binomial edge ideals of graphs. These ideals draw considerable attention due to their nice structure and properties. To begin, let us define binomial edge ideals.
\par 

Let $G$ be a simple graph with vertex set $V(G)=[n]=\{1,\ldots,n\}$ and edge set $E(G)$, where $E(G)$ is a collection of subsets of $V(G)$ with exactly two elements. Then the \textit{binomial edge ideal} of $G$, denoted by $J_{G}$, is an ideal of $S$ defined as
$$J_{G}:=\big<f_{ij} = x_{i}y_{j}-x_{j}y_{i}\mid \{i,j\}\in E(G)\,\, \text{with}\,\, i<j\big>,$$
where $S$ is the polynomial ring $\mathbb{K}[x_{1},\ldots,x_{n},y_{1},\ldots,y_{n}]$ over a field $\mathbb{K}$.
\par 

The study of binomial edge ideals began in 2010 with independent studies by Herzog et al. \cite{hhhrkara} and Ohtani \cite{ohtani}. Binomial edge ideals can be interpreted as a broader form of the determinantal ideals associated with the $2$-minors of a $2\times n$ matrix of indeterminates. One key motivation behind studying these ideals stems from their relevance to algebraic statistics. Mainly, these ideals appear in the study of `conditional independence statements' as demonstrated in \cite[Section 4]{hhhrkara}. 
\par 

Among all the studies regarding binomial edge ideals, two investigations stand out in their own way: (i) the Cohen-Macaulay property and (ii) the Castelnuovo-Mumford regularity (in short, regularity) of binomial edge ideals. Unlike monomial edge ideals of graphs, to our best knowledge, there are no instances in the literature where the Cohen-Macaulay property or the regularity of binomial edge ideals depends on the characteristics of the base field $\mathbb{K}$. 
\par 

A significant body of research has focused on understanding Cohen-Macaulay binomial edge ideals (\cite{mont20}, \cite{bms_cmbip}, \cite{acc}, \cite{ehh_cmbin}, \cite{hhhrkara}, \cite{ks_cmunm15}, \cite{lmrr23}, \cite{raufrin14}, \cite{rin_cactus19}, \cite{rin_smaldev13}, \cite{ssgirth24}, \cite{sswhisker}), yet a comprehensive combinatorial characterization is still open. Notably, a recent paper \cite{acc} has made significant progress in this direction. Specifically, the authors of \cite{acc} introduced the concept of \textit{accessible} graphs (\cite[Definition 2.2]{acc}) to give a possible combinatorial interpretation of Cohen-Macaulay binomial edge ideals. They proved that if $J_{G}$ is Cohen-Macaulay, then $G$ must be accessible, and they proposed \cite[Conjecture 1.1]{acc} regarding the converse. This conjecture has been confirmed for several graph classes, including bipartite and chordal (see \cite{acc}, \cite{bmrs_smallgraphs24}, \cite{lmrr23}, \cite{sswhisker}).
%Additionally, \cite{ssgirth24} and \cite{sswhisker} have shown that focusing on `\textit{biconnected graphs attached with some whiskers}' is sufficient to classify all Cohen-Macaulay binomial edge ideals.
\par

 In addition to investigating the Cohen-Macaulay property of binomial edge ideals, considerable efforts have been directed towards determining the regularity of binomial edge ideals for various classes of graphs (see \cite{das_binomsurvey23} and \cite{adamla23} for a comprehensive survey). Due to \cite[Conjecture 1.1]{acc}, we believe that the regularity of Cohen-Macaulay binomial edge ideals does not depend on the characteristic of the underlying field and may have some nice combinatorial description. To date, we know the exact formulae for the regularity of Cohen-Macaulay binomial edge ideals of block graphs \cite[Corollary 3.3]{jnr19} and bipartite graphs \cite[Theorem 4.7]{jk_regcmbip19}. Although there are some other classes of Cohen-Macaulay binomial edge ideals, which have been combinatorially classified, their regularity remains unexplored. This includes classes such as chordal graphs \cite{acc}, traceable graphs \cite{acc}, a chain of cycles attached with some whiskers \cite{lmrr23}, and $r$-regular $r$-connected block with whiskers \cite{sswhisker}.
\par 

In this paper, we explicitly give combinatorial formulae for the regularity of two classes of Cohen-Macaulay binomial edge ideals: (i) chain of cycles attached with some whiskers (classified in \cite{lmrr23}) and (ii) the class given in \cite{sswhisker}, which includes $r$-regular $r$-connected blocks attached with some whiskers. Mainly, we use the regularity lemma, the regularity formula for decomposable graphs and the regularity of certain block graphs to establish the upper bound of the regularity of the desired classes. In \cite[Corollary 2.7]{cv20}, Conca and Varbaro showed that for a graded ideal in a polynomial ring with a square-free initial ideal for some term order, the regularity of that ideal and its initial ideal are equal. We use this result by examining the initial ideals of binomial edge ideals to get the required lower bound of the regularity for the second class. The paper is structured as follows:
\par 

Section \ref{preli} is devoted to discussing the necessary prerequisites. Let $B$ be a chain of cycles (see \Cref{defchain}), and $\ov{B}$ be the graph after attaching some whisker to $B$. Then it has been proved in \cite[Theorem 3]{lmrr23} that $J_{\ov{B}}$ is Cohen-Macaulay if and only if $\ov{B}$ satisfies \cite[Setup 1]{lmrr23}. In \Cref{secchain}, we provide the exact combinatorial formula of the regularity for Cohen-Macaulay $J_{\ov{B}}$. The formula is based on the number of blocks in an induced block graph of $\ov{B}$ whose binomial edge ideal is Cohen-Macaulay. Let $H$ be a block graph (see \Cref{defblock}). Then $J_H$ is Cohen-Macaulay if and only if no vertex of $H$ belongs to more than two maximal cliques (see \cite[Theorem 1.1]{ehh_cmbin}). Again, if $H$ is a block graph with $J_H$ Cohen-Macaulay, then by \cite[Corollary 3.3]{jnr19}, we have $\mathrm{reg}(S/J_H)=b(H)$, where $b(H)$ denote the number of blocks (maximal cliques) in $H$. Now, for any simple graph $G$, let us define the following new invariant:
$$b(G):=\max\{b(H)\mid H\text{ is an induced block graph of }G \text{ with }J_{H}\text{ Cohen-Macaulay}\}.$$ 
The following is the main theorem of \Cref{secchain}.
\medskip

\noindent\textbf{Theorem \ref{thmchain}.} \textit{Let $G=\ov{B}$ with $J_G$ Cohen-Macaulay, where $B$ is a chain of cycles. Then 
$$\reg(S/J_G)=b(G).$$
}

\noindent In \Cref{secreg-connected}, we find the regularity of the class of Cohen-Macaulay binomial edge ideals given in \cite{sswhisker}. Let $K_m$ denote the complete graph on $m$ vertices. Corresponding to two complete graphs $K_m$ and $K_n$ and a positive integer $2\leq r\leq \min\{m,n\}$, the author in \cite{sswhisker} construct a graph $\ov{K_m\star_{r} K_n}$ such that $J_{\ov{K_m\star_{r} K_n}}$ is Cohen-Macaulay. Particularly, when $m=n=r$ it gives any $r$-regular $r$-connected block whose binomial edge ideal is Cohen-Macaulay after attaching some whiskers to it. The following gives the formula for the regularity of $S/J_G$, where $G=\ov{K_m\star_{r} K_n}$.
\medskip

\noindent\textbf{\Cref{thmstar2}, \ref{thmregcon}.} \textit{Let $G = \ov{K_m \star_r K_n}$. Then we have
\begin{align*}
    \reg(S/J_G)=
    \begin{cases}
        4 & \text{if $r=2$ and $m>2$ or $n>2$}\\
        3 & \text{if $r=m=n=2$}\\
        2r-1 & \text{if $r>2$}.
    \end{cases}
\end{align*}
}

\section{Preliminaries}\label{preli}

 We assume all graphs are simple and finite. Let $G$ be a graph. A graph $H$ is said to be an \textit{induced subgraph} of $G$ if $V(H)\subset V(G)$ and $E(H)=\{e\in E(G)\mid e\subset V(H)\}$. For $T\subset V(G)$, we write $G\setminus T$ to denote the induced subgraph of $G$ on the vertex set $V(G)\setminus T$, and also, we write $G[T]$ to mean the induced subgraph of $G$ on the vertex set $T$. For a vertex $v\in V(G)$, we say $\mathcal{N}_{G}(v)=\{u\in V(G)\mid \{u,v\}\in E(G)\}$ the \textit{neighbour set} of $v$ in $G$. If $\mathcal{N}_{G}(v)=\{u\}$, then $\{u,v\}\in E(G)$ 
is called a \textit{whisker} attached to $u$. The \textit{degree} of a vertex $v\in V(G)$ is $\vert \mathcal{N}_{G}(v)\vert$. A \textit{cycle} of length $n$ is a connected graph on $n$ vertices such that the degree of each vertex is $2$. A \textit{path} from $u$ to $v$ of length $n$ 
in $G$ is a sequence of vertices $u=v_{0},\ldots,v_{n}=v\in V(G)$, such that $\{v_{i-1},v_{i}\}\in E(G)$ for each $1\leq i\leq n$, and 
$v_{i}\neq v_{j}$ if $i\neq j$.
\medskip

A graph is said to be \textit{complete} if there is an edge between every pair of vertices, and we denote the complete graph on $n$ vertices by $K_{n}$. A \textit{clique} of a graph $G$ is an induced complete subgraph of $G$. A vertex $v\in V(G)$ is called a \textit{free vertex} of $G$ if the induced subgraph of $G$ on $\mathcal{N}_{G}(v)$ is complete. A graph $G$ is said to be a \textit{gluing} of $G_{1}$ and $G_{2}$ at the vertex $v$, if $G=G_{1}\cup G_{2}$ with $V(G_{1})\cap V(G_{2})=\{v\}$ such that $v$ is a free vertex of both $G_{1}$ and $G_{2}$. The notion of gluing was introduced in \cite{raufrin14} to study the Cohen-Macaulay property of binomial edge ideals. In literature, people also use to say $G$ is \textit{decomposable} into $G_{1}$ and $G_{2}$ instead of saying gluing and write $G=G_1\cup_{v} G_2$ to mean $G$ is a gluing of $G_1$ and $G_2$ at the vertex $v$. A vertex $v\in V(G)$ is said to be a \textit{cut vertex} of $G$ if the removal of $v$ from $G$ increases the number of connected components.

\begin{definition}\label{defblock}{\rm
    A graph $G$ is said to be \textit{chordal} if $G$ has no induced cycle of length greater than $3$. If any two distinct maximal cliques of a chordal graph $G$ intersect in at most one vertex, then $G$ is known as a \textit{block} graph. By a \textit{block} in a block graph, we mean a maximal clique.
    }
\end{definition}

\begin{theorem}[{\cite[Theorem 1.1]{ehh_cmbin}}]\label{thmcmblock}
    Let $G$ be a block graph. Then $J_G$ is Cohen-Macaulay if and only if each vertex of $G$ is the intersection of at most two maximal cliques
\end{theorem}

\begin{remark}\label{remregblock}{\rm
    Let $G$ be a gluing of two graphs $G_1$ and $G_2$. Then by \cite[Theorem 3.1]{jnr19}, $$\reg(S/J_G)=\reg(S_1/J_{G_1})+\reg(S_2/J_{G_2}),$$
    where $S_1$ and $S_2$ are the corresponding polynomial rings of $J_{G_1}$ and $J_{G_2}$. As an application \cite[Corollary 3.3]{jnr19}, we get $\reg(S/J_G)=b(G)$ if $G$ is a block graph with $J_G$ Cohen-Macaulay. 
    }
\end{remark}

\begin{definition}{\rm
Let $v\not\in V(G)$. Then the \textit{cone} of $v$ on $G$, denoted by $\mathrm{cone}(v,G)$, is the graph with vertex set $V(G)\cup \{v\}$ and edge set $E(G)\cup \{\{u,v\}\mid u\in V(G)\}$.
}
\end{definition}

\noindent The study of binomial edge ideals of the cone on a graph was initiated in \cite{raufrin14}. Later, the concept of cone had been generalized to the join of two graphs in \cite{mkreg3}.
\medskip

Let $G$ be a graph and $v$ be any vertex of $G$. Let us define a graph $G_v$ such that
$$V(G_v)=V(G)\,\text{ and }\, E(G_v)=E(G)\cup \{\{i,j\}\mid i,j\in\mathcal{N}_{G}(v), i\neq j\}.$$
The graph $G_v$ plays an important role in the study of binomial edge ideal. In particular, thanks to \cite[Lemma 4.8]{ohtani}, we have the following exact sequence, which has been used extensively in many papers.
\begin{align}\label{eqexact}
   0 \longrightarrow \frac{S}{J_G} \longrightarrow \frac{S}{\langle J_{G \setminus \{v\}},x_v,y_v\rangle} \oplus \frac{S}{J_{G_v}} \longrightarrow \frac{S}{\langle J_{G_v \setminus \{v\}},x_v,y_v\rangle}\longrightarrow 0. 
\end{align}

\noindent\textbf{Note:} By saying an ideal $I$ of $R$ is Cohen-Macaulay, we mean the quotient ring $R/I$ is Cohen-Macaulay. Whenever we write $S/J_G$, we mean $S$ to be the corresponding polynomial ring of the binomial edge ideal $J_G$.

\section{Regularity of chain of cycles with whiskers}\label{secchain}

In this section, we determine the combinatorial formula for the regularity of the class of Cohen-Macaulay binomial edge ideals provided in \cite{lmrr23}. In particular, we find the regularity of the Cohen-Macaulay binomial edge ideal of a chain of cycles with whiskers.\par 

A subgraph $H$ of $G$ spans $G$ if $V(H) = V(G)$. In a connected graph $G$, a chord of a tree $H$ that spans $G$ is an edge of $G$ not in $H$. The number of chords of any spanning tree of a connected graph $G$, denoted by $m(G)$, is called the \textit{cycle rank} of $G$, and it is
given by $m(G) = \vert E(G)\vert -\vert V(G)\vert + 1$.

\begin{definition}[{\cite[Definition 3]{lmrr23}}]\label{defchain}
    {\rm Let $B$ be a block with $m(B) = r$ such that $B=\cup_{i=1}^{r}D_i$, where $D_i$'s are
cycles, $E(D_i) \cap E(D_{i+1}) = E(P)$, where $P$ is a path, and for all $j\neq i-1,i,i+1$, $E(D_i) \cap E(D_j)=\emptyset$. We call $B$ a chain of cycles.
}
\end{definition}

\begin{definition}{\rm
    By a \textit{block} $B$ in a graph $G$, we mean a maximal biconnected induced subgraph of $G$. Let $B$ be a block of a graph $G$ and $W=\{w_1,\ldots,w_r\}$ be the set of cut vertices of $G$ belonging to $V(B)$. Then we define a graph $\ov{B}$, called \textit{blcok with whiskers} with respect to $G$, as follows:
    \begin{enumerate}
        \item[$\bullet$] $V(\ov{B})=V(B)\cup \{f_1,\ldots,f_r\}$;
        \item[$\bullet$] $E(\ov{B})=E(B)\cup \{\{w_i,f_i\}\mid 1\leq i\leq r\}.$
    \end{enumerate}
The edges $\{w_i,f_i\}$ are called whiskers attached to $B$. Moreover, when the graph $G$ is not mentioned, by a `block with whiskers', we mean a biconnected graph $B$ with some whiskers attached to it and denote it by $\ov{B}$.
    }
\end{definition}

In \cite{lmrr23}, it has been completely characterized when $J_{\ov{B}}$ is Cohen-Macaulay for a chain of cycles $B$. If $B$ is a chain of cycles, then $J_{\ov{B}}$ is Cohen-Macaulay if and only if $\ov{B}$ satisfies \cite[Setup 1]{lmrr23}. Due to \cite[Proposition 3.2]{ssgirth24}, if we replace the $D_1$ in \cite[Setup 1]{lmrr23} by any complete graph, then also $J_{\ov{B}}$ is Cohen-Macaulay. Thus, we will consider a bigger class of Cohen-Macaulay binomial edge ideals, which contains the class chain of cycles with whiskers. In particular, we consider those blocks with whiskers, which satisfy the following setup (see \Cref{fig1}).

\begin{setup}\label{setup}
Let $\ov{B}$ be a block with whiskers, where $B = \cup_{i=1}^r D_i$ is a chain of cycle with a possible replacement of $D_1$ by a complete graph, satisfying the following properties:
\begin{enumerate}
\item[(i)] $D_1\in\{K_n,C_4\mid n\geq 3\}$ and $D_i \in \{C_3,C_4\}$ for each $i\geq 2$;
\item[(ii)] If $D_i = C_4$, then $D_{i+1} = C_3$;
\item[(iii)] $E(D_i) \cap E(D_{i+1}) = \{w_i,u_i\}$, where $w_i$ is a cut vertex and $u_i$ is not a cut vertex;
\item[(iv)] $\{w_i, w_{i+1}\} \in E(D_{i+1})$ (resp. $\{u_i, u_{i+1}\} \in E(D_{i+1}))$ or $w_i = w_{i+1}$ (resp. $u_i = u_{i+1}$);
\item[(v)] If $D_1 = C_4$ with $V(D_1) = \{w_0,w_1,u_0,u_1\}$ and $ \{w_0,w_1\}, \{u_0,u_1\} \in E(D_1)$, then $w_0$ and $w_1$ are cut vertices, whereas $u_0$ and $u_1$ are not cut vertices;
\item[(vi)] If $D_r = C_4$ with $V(D_r) = \{w_{r-1}, w_{r}, u_{r-1}, u_{r}\}$ and $\{w_{r-1}, w_{r}\}, \{u_{r-1}, u_{r}\} \in E(D_r)$, then $w_{r-1}$ and $w_{r}$ are cut vertices, whereas $u_{r-1}$ and $u_{r}$ are not cut vertices;
\item[(vii)] If $v \in V(B)$ is such that $v$ belongs to four $D_{i}$'s or $v$ belongs to three $D_{i}$'s with $v$ a vertex of a $C_4$, then $v$ is a cut vertex of $\ov{B}$.
\end{enumerate}
\end{setup}

%\item[(vii)] \textcolor{blue}{If $v \in V(B)$ with $\mathrm{deg}(v) \geq 5$ or $\mathrm{deg}(v) \geq 4$ with $v$ a vertex of a $C_4$, then $v$ is a cut vertex.}

\begin{figure}[h]
\begin{tikzpicture}
 [scale=1]
 
\filldraw[black] (0,0) circle (2pt)node[anchor=east]{};
\filldraw[black] (2,0) circle (2pt)node[anchor=south]{};
\filldraw[black] (0,2) circle (2pt)node[anchor=north]{};
\filldraw[black] (2,2) circle (2pt)node[anchor=south]{};
\filldraw[black] (4,0) circle (2pt)node[anchor=north]{};
\filldraw[black] (4,2) circle (2pt)node[anchor=north]{};
\filldraw[black] (5,0) circle (2pt)node[anchor=south]{};
\filldraw[black] (6,0) circle (2pt)node[anchor=south]{};
\filldraw[black] (6,2) circle (2pt)node[anchor=south]{};
\filldraw[black] (7,0) circle (2pt)node[anchor=south]{};
\filldraw[black] (8,0) circle (2pt)node[anchor=south]{};
\filldraw[black] (8,2) circle (2pt)node[anchor=south]{};

\filldraw[black] (2,3.5) circle (2pt)node[anchor=south]{};
\filldraw[black] (4,3.5) circle (2pt)node[anchor=south]{};
\filldraw[black] (6,3.5) circle (2pt)node[anchor=south]{};
\filldraw[black] (8,3.5) circle (2pt)node[anchor=south]{};

\draw[black] (0,0) -- (2,0);
\draw[black] (0,0) -- (0,2);
\draw[black] (0,0) -- (2,2);
\draw[black] (2,0) -- (0,2);
\draw[black] (2,0) -- (2,2);
\draw[black] (2,2) -- (0,2);
\draw[black] (4,0) -- (4,2);
\draw[black] (2,0) -- (4,0);
\draw[black] (2,2) -- (4,2);
\draw[black] (4,2) -- (5,0);
\draw[black] (4,0) -- (5,0);
\draw[black] (5,0) -- (6,0);
\draw[black] (4,2) -- (6,0);
\draw[black] (4,2) -- (6,2);
\draw[black] (6,0) -- (6,2);
\draw[black] (7,0) -- (6,2);
\draw[black] (7,0) -- (6,0);
\draw[black] (8,0) -- (8,2);
\draw[black] (8,0) -- (7,0);
\draw[black] (8,2) -- (6,2);

\draw[black] (2,3.5) -- (2,2);
\draw[black] (4,3.5) -- (4,2);
\draw[black] (6,3.5) -- (6,2);
\draw[black] (8,3.5) -- (8,2);

\end{tikzpicture}
\caption{A graph $\ov{B}$ satisfying \Cref{setup}.}

\label{fig1}
\end{figure}

\noindent \textbf{Note:} Due to \Cref{remregblock}, we may assume $G$ is not a decomposable graph satisfying \Cref{setup}, i.e. $G$ is not a complete graph. Corresponding to a cut vertex $w$ of $G$, we denote the other vertex of the whisker attached to $w$ by $f_w$.
\medskip

%\textcolor{blue}{For a graph $\ov{B}$ defined in the above setup, the binomial edge ideal $J_{\bar{B}}$ is Cohen-Macaulay. Note that in the above setup, if we replace $D_1$ by any complete graph such that $D_1 \cap D_2$ is an edge, then also $J_{\bar{B}}$ is Cohen-Macaulay  In this article, we consider the graphs which satisfy setup-1 with $D_1 \in \{C_4, K_n\}$, where $n \geq 3$. }

Let $G = \ov{B}$, where $\ov{B}$ satisfies Setup \ref{setup}. We want to find the regularity of $S/J_G$, and for that purpose, we will define some terminologies to get a better picture of our proof. Define 
$$S(G):= \{D_i \mid D_i ~ \text{contains exactly one cut vertex of} ~ G \}.$$  
Note that every member of $S(G)$ is a complete graph, and at most, one member of $S(G)$ can be a complete graph on more than three vertices. From the structure of $G$ as mentioned in Setup \ref{setup}, it is clear that after removing some cut vertices from $G$, we will get a block graph whose binomial edge ideal is Cohen-Macaulay. For example, remove all the cut vertices, and we will left with a path graph or gluing of a path graph and a complete graph. Thus, we define the following: 
\begin{align*}
    \mathbf{CMB}(G) := \{ H \mid\,\, & H = G \setminus W ~ \text{for some set of cut vertices $W$ of $G$}\\ 
    & \text{and $H$ is a block graph with $J_H$ Cohen-Macaulay} \}.
\end{align*}
    
\noindent For any block graph $H$, let $b(H)$ denote the number of blocks in $H$. Now, let us define $b(G)$ for any graph $G$ satisfying Setup \ref{setup} as follows:
 $$b(G):= \mathrm{max}\{b(H) \mid H \in \mathbf{CMB}(G)\}.$$

\begin{remark}\label{remark1}{\rm
Let $H \in \mathbf{CMB}(G)$ and $H = G\setminus W$ for some collection $W$ of cut vertices of $G$. Then, the following are easy to verify from \Cref{thmcmblock} and \Cref{setup}.
\begin{enumerate}
\item[(i)] If $w \in V(G)$ is a cut vertex of $G$ such that $w$ belongs to more than one member of $S(G)$, then $w \in W$.
\item[(ii)] If $D_i= C_4$ for some $i$, then at least one of $w_{i-1}$ or $w_{i}$ belongs to $W$.
\item[(iii)] If $D_i=C_3$ contains two cut vertices of $G$, then $D_{i-1},D_{i+1}\in S(G)$ due to the condition (vii) of Setup \ref{setup}. Therefore, if $w$ is a cut vertex of $G$ with $w \notin W$, then $N_G(w) \cap W(G) \subset W$ as $H\in \mathbf{CMB}(G)$, where $W(G)$ denotes the set of all cut vertices of $G$ (i.e., if $w$ is a cut vertex of $G$ which also belongs to $H$, then the neighbour cut vertices of $w$ should belong to $W$).  
\end{enumerate}
}
\end{remark}

\begin{figure}[h]
\centering
 \begin{minipage}{.5\textwidth}
 \centering
\begin{tikzpicture}
 [scale=1]
 
\filldraw[black] (0,0) circle (2pt)node[anchor=east]{$u_0$};
\filldraw[black] (2,0) circle (2pt)node[anchor=south]{};
\filldraw[black] (0,2) circle (2pt)node[anchor=east]{$w=w_0$};
\filldraw[black] (2,2) circle (2pt)node[anchor=west]{$w_1$};
\filldraw[black] (4,0) circle (2pt)node[anchor=north]{};

\filldraw[black] (0,3.5) circle (2pt)node[anchor=east]{$f_{w_0}$};
\filldraw[black] (2,3.5) circle (2pt)node[anchor=south]{};

\filldraw[black] (4.5,1) circle (1pt)node[anchor=south]{};
\filldraw[black] (5,1) circle (1pt)node[anchor=south]{};
\filldraw[black] (5.5,1) circle (1pt)node[anchor=south]{};

\draw[black] (0,0) -- (2,0);
\draw[black] (0,0) -- (0,2);
\draw[black] (2,0) -- (2,2);
\draw[black] (2,2) -- (0,2);
\draw[black] (2,2) -- (4,0);
\draw[black] (2,0) -- (4,0);

\draw[black] (2,3.5) -- (2,2);
\draw[black] (0,3.5) -- (0,2);

\end{tikzpicture}
\caption{A graph $\ov{B}$ considered in \textbf{Case-1} of \Cref{lembgw<bg}.}
\label{fig2}
\end{minipage}%
\begin{minipage}{.5\textwidth}
    \centering
\begin{tikzpicture}
 [scale=1]
 
\filldraw[black] (0,0) circle (2pt)node[anchor=east]{};
\filldraw[black] (2,0) circle (2pt)node[anchor=south]{};
\filldraw[black] (0,2) circle (2pt)node[anchor=north]{};
\filldraw[black] (2,2) circle (2pt)node[anchor=west]{$w=w_1$};
\filldraw[black] (4,0) circle (2pt)node[anchor=west]{$u=u_3$};
\filldraw[black] (3,0) circle (2pt)node[anchor=north]{};

\filldraw[black] (2,3.5) circle (2pt)node[anchor=west]{$f_{w}$};

\filldraw[black] (4.5,1) circle (1pt)node[anchor=south]{};
\filldraw[black] (5,1) circle (1pt)node[anchor=south]{};
\filldraw[black] (5.5,1) circle (1pt)node[anchor=south]{};

\draw[black] (0,0) -- (2,0);
\draw[black] (0,0) -- (0,2);
\draw[black] (2,0) -- (2,2);
\draw[black] (2,2) -- (0,2);
\draw[black] (2,2) -- (0,0);
\draw[black] (2,0) -- (0,2);
\draw[black] (2,2) -- (4,0);
\draw[black] (2,0) -- (3,0);
\draw[black] (4,0) -- (3,0);
\draw[black] (2,2) -- (3,0);

\draw[black] (2,3.5) -- (2,2);

\end{tikzpicture}
\caption{A graph $\ov{B}$ considered in \textbf{Case-2} of \Cref{lembgw<bg}.}

\label{fig3}
\end{minipage}
\end{figure}

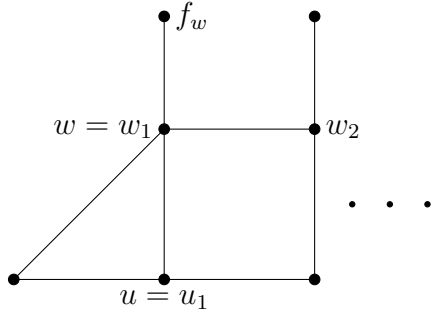
\begin{figure}[h]
\begin{tikzpicture}
 [scale=1]
 
\filldraw[black] (0,0) circle (2pt)node[anchor=east]{};
\filldraw[black] (2,0) circle (2pt)node[anchor=north]{$u=u_1$};
\filldraw[black] (2,2) circle (2pt)node[anchor=east]{$w=w_1$};
\filldraw[black] (4,0) circle (2pt)node[anchor=west]{};
\filldraw[black] (4,2) circle (2pt)node[anchor=west]{$w_2$};
\filldraw[black] (4,3.5) circle (2pt)node[anchor=north]{};

\filldraw[black] (2,3.5) circle (2pt)node[anchor=west]{$f_{w}$};

\filldraw[black] (4.5,1) circle (1pt)node[anchor=south]{};
\filldraw[black] (5,1) circle (1pt)node[anchor=south]{};
\filldraw[black] (5.5,1) circle (1pt)node[anchor=south]{};

\draw[black] (0,0) -- (2,0);
\draw[black] (2,0) -- (2,2);
\draw[black] (2,2) -- (0,0);
\draw[black] (2,2) -- (4,2);
\draw[black] (2,0) -- (4,0);
\draw[black] (4,0) -- (4,2);

\draw[black] (2,3.5) -- (2,2);
\draw[black] (4,3.5) -- (4,2);

\end{tikzpicture}
\caption{A graph $\ov{B}$ considered in \textbf{Case-3} of \Cref{lembgw<bg}.}
\label{fig4}
\end{figure}

\begin{lemma}\label{lembgw<bg}
Let $G$ be a graph satisfying Setup \ref{setup}. If $w$ is the first cut vertex of $G$ (i.e., $w \in V(D_1)$ and if $D_1 \cong C_4$, then $w = w_0$), then $b(G_w) < b(G).$
\end{lemma}
\begin{proof}
We will break the proof into three cases:\par 

\noindent \textbf{Case-1:} Suppose $D_1 = C_4$ (see \Cref{fig2}). Let $b(G_w) = b(H)$ for some $H \in \mathbf{CMB}(G_w)$. Thus, $H$ has to be of the form $G_w \setminus W$, where $W$ is a collection of cut vertices of $G_w$. Note that $w_1$ belongs to more than one member of $S(G_w)$. Thus, by Remark \ref{remark1}(i), we have $w_1 \in W$. Therefore, $H$ can be viewed as a gluing of two graphs $H_1$ and $H_2$, where $H_1 = G_w[\{u_0,w_0,f_{w_0}\}] \cong C_3$ and $V(H_1) \cap V(H_2) = \{u_0\}$. Now, consider the graph $H' = G \setminus W$. From the structure of $G$ and $G_w$ it is easy to observe that $H'$ is a gluing of $H_1'$ and $H_2$ at $u_{0}$, where $V(H_1') = \{u_0,w_0,f_{w_0}\}$ and $E(H_1') = \{\{u_0,w_0\}, \{w_0,f_{w_0}\}\}$. This means that $H_1' = G[\{u_0,w_0,f_{w_0}\}]$ and $V(H_1') \cap V(H_2) = \{u_0\}$. By \Cref{thmcmblock}, $H'$ is a block graph with $J_{H'}$ Cohen-Macaulay, and thus, $H' \in \mathbf{CMB}(G)$. Also, we have $b(H') = b(H'_{1})+b(H_2)=2+b(H_2)=1+b(H_1)+b(H_2)= b(H)+1$. This implies $b(G) \geq b(H') = b(H)+1 = b(G_w)+1$. Hence, we get $b(G) > b(G_w)$.\par 

\noindent \textbf{Case-2:} Suppose $w$ belongs to more than one members of $S(G)$ (see \Cref{fig3}). Then 
$$\{D_i\mid w\in V(D_i)\}\cap S(G)=\{D_1, \ldots, D_k\}$$
for some $k \geq 2$.
By definition, the induced subgraph of $G_w$ on $N_G[w]$ is a complete graph, i.e. $G_{w}[N_G[w]]$ is a complete graph. Let $b(G_w) = b(H)$ for some $H \in \mathbf{CMB}(G_w)$. Then $H$ has to be of the form $G_w \setminus W$ for some set of cut vertices $W$ of $G_w$. If $G_w$ is complete, then $W = \emptyset$ and $H = G_w$. In this case, $b(G_w) = b(H) = 1$ and $b(G) = b(G \setminus \{w\}) = k$. Therefore, $b(G) > b(G_w)$ as $k \geq 2$. Now, let us assume that $G_w$ is not a complete graph. Note that $w_1 = w_2 = \ldots = w_k = w$ in $G$, and so, $w_{k+1}$ is the first cut vertex of $G_w$. From the structure of $G$, it is easy to observe that $G$ and $G_{w}$ both are indecomposable graphs. Thus, for any choice of $D_{k+1}$, $w_{k+1}$ belongs to more than one members of $S(G_w)$. Therefore, we have $w_{k+1} \in W$ due to Remark \ref{remark1}(i). Thus, $H$ can be viewed as a gluing of $H_1$ and $H_2$ at the vertex $u=u_{k}$, where $H_1$ is a complete graph on $N_G[w]\setminus \{w_{k+1}\}$. Now, we consider the graph $H' = G \setminus (W \cup \{w\}).$ Then $H'$ is a gluing of $H_1'$ and $H_2$, where $H_1'$ is $G[N_G[w] \setminus \{w, f_w, w_{k+1}\}].$ This means that $H_1'$ is a gluing of $k$ complete graphs among which at least $k-1$ complete graphs are isomorphic to $K_2$. Therefore, $H' \in \mathbf{CMB}(G)$ with $b(H') = b(H) + k - 1.$ Since $k \geq 2$, $b(H') > b(H)$. Hence, we have $b(G) > b(G_w)$.\\
\noindent \textbf{Case-3:} Suppose $D_1 \neq C_4$ and $w$ does not belong to more than one member of $S(G)$ (see \Cref{fig4}). Therefore, from the structure of $G$ as in \Cref{setup}, it is easy to see that $D_1 \in S(G)$ and $D_2$ is either a $C_4$ or a $C_3$ which contains two cut vertices of $G$. Now, consider $G_w$ and let $b(G_w) = b(H) $ for some $H \in \mathbf{CMB}(G_w)$. Then $H$ has to be of the form $G_w \setminus W$, where $W$ is a collection of cut vertices of $G_w$. Since $G$ is not a decomposable graph, there exists $D_2$ and a cut vertex $w_2$ of $G$ such that $w_2 \in N_G(w)$. Therefore, $w_2$ should belong to more than one members of $S(G_w)$, and hence, $w_2 \in W$ by \Cref{remark1}(i). Let $u$ be the non-cut vertex belong to $E(D_1)\cap E(D_2)$, i.e. $u=u_1$. Then $H$ is a gluing of $H_1$ and $H_2$ at the vertex $u$, where $H_1= K_n$ for some $n\geq 4$ (precisely, $n=\vert V(D_1)\vert +1=\vert N_{G}[w]\vert -1$) and $H_2$ is a block graph with $J_{H_2}$ Cohen-Macaulay. Now, we consider the graph $H' = G \setminus W$. Then $H'$ is the gluing of $H_1'$ and $H_2$ at $u$, where $H_1' = K_2 \cup_w D_1$. Since $D_1$ is a complete graph, $H_{1}'$ is a block graph with $J_{H_{1}'}$ Cohen-Macaulay. Hence, $H'$ is a block graph with $J_{H'}$ Cohen-Macaulay by Theorem \ref{thmcmblock}, i.e. $H'\in\mathbf{CMB}(G)$. Now, $b(H') = b(H_{1}')+b(H_2)= 2+b(H_2)$. Also, $b(H) = b(H_1)+b(H_2)=1+b(H_2)$. Therefore, we have $b(G) \geq b(H') > b(H) = b(G_w)$.
\end{proof}

\begin{theorem}\label{thmchain}
Let $G = \ov{B}$ be the graph satisfying \Cref{setup}. Then $\mathrm{reg}(S/J_G) = b(G)$. In particular, if $B$ is a chain of cycles such that $J_{\ov{B}}$ is Cohen-Macaulay, then $\reg(S/J_{\ov{B}})=b(\ov{B})$.
\end{theorem}

\begin{proof}
Let $w$ be the first cut vertex of $G$, i.e. $w = w_0$ when $D_1=C_4$, otherwise $w = w_1$. 
By \Cref{lembgw<bg}, we have $b(G) > b(G_w)$.  First, we will show that $\reg(S/J_G)\leq b(G)$ by using induction on the number of cut vertices of $G$. Let the number of cut vertices of $G$ be $1$. Then $G=\mathrm{cone}(w, H)$, where $H=H_1\cup H_2$, $H_1$ is a block graph with $J_{H_1}$ Cohen-Macaulay and $H_2$ is the isolated vertex $f_w$. Then it is easy to see that $b(G)=b(G\setminus \{w\})=b(H_1)$. Therefore, by \cite[Theorem 2.1]{mkreg3} and \Cref{remregblock}, we have $\reg(S/J_G)=\reg(S/J_{H_1})=b(H_1)=b(G)$. Now, let us assume $G$ has more than $1$ cut vertex. First, we consider the graph $G_w$. Since $G_w$ has one less cut vertices than $G$ and $G_w$ also satisfies \Cref{setup}, by induction hypothesis and \Cref{lembgw<bg},
\begin{equation}\label{eq1}
 \mathrm{reg}(S/J_{G_w}) \leq b(G_w)< b(G).
 \end{equation}
Let $G_w=\ov{B'}$, where $B'=\cup_{i=1}^{r'}D'_{i}$. Since $w$ is a free vertex in $G_w$ and $D'_{1}$ is a complete graph on at least $4$ vertices containing $w$, it is easy to verify that $H' \in \mathbf{CMB}(G_w \setminus \{w\})$ if and only if $H'=H\setminus \{w\}$ for some $H \in \mathbf{CMB}(G_w)$. Let $b(G_w) = b(H)$ for some $H \in \mathbf{CMB}(G_w)$. Then, the block of $H$, which contains the vertex $w$ (say $B_1$), is a complete graph on at least three vertices. If $H = G_w \setminus W$ for some set of cut vertices $W$ of $G_w$, then $(G_w \setminus \{w\}) \setminus W = H' \in \mathbf{CMB}(G_w \setminus \{w\})$. If $B_1$ is the block of $H$ containing $w$, then $B_1 \setminus \{w\}$ is a block of $H'$ and $B_1 \setminus \{w\}$ can not be empty. Therefore, we have $b(H') = b(H)$, and thus, $b(G_w \setminus \{w\}) \geq b(G_w)$.
Let $b(G_w \setminus \{w\}) = b(H')$ for some $H' \in \mathbf{CMB}(G_w \setminus \{w\})$. Then $H' = H \setminus w$ for some $H \in \mathbf{CMB}(G_w)$. Thus, we get $b(H) = b(H')$ because  $w$ is a free vertex of $H$, and the block containing $w$ in $H$ has at least three vertices. Therefore, $b(G_w) \geq b(G_w \setminus \{w\})$, and hence, $b(G_w) = b(G_w \setminus \{w\})$. Since $G_{w}\setminus\{w\}$ has less number of cut vertices than $G$, by induction hypothesis and \Cref{lembgw<bg}, it follows that 
\begin{equation}\label{eq2}
\mathrm{reg}(S/\langle J_{G_w \setminus \{w\}}, x_w,y_w \rangle) \leq b(G_w \setminus \{w\}) = b(G_w) < b(G).
\end{equation}
Now, we will show that $ \mathrm{reg}(S/\langle J_{G \setminus \{w\}}, x_w,y_w \rangle) \leq b(G)$. Since the number of cut vertices of $G$ is more than $1$, there exists $k$ such that $w$ belongs to $D_1, \ldots, D_k$ and $w \not\in D_{k+1}$ when $k<r$ or $D_k=C_4$ when $k=r$. If $k<r$, then we consider $u$ to be the non-cut vertex belongs to $D_k \cap D_{k+1}$, and if $k=r$, then we consider $u=u_r$. Now, due to the structure of $G$ given in \Cref{setup} (especially, (ii) and (vii) of \Cref{setup}), one can observe that $G \setminus \{w\}$ is a gluing of $H_1$ and $H_2$ at $u$, where $H_1$ is a gluing of some complete graphs among which only one complete graph can have more than two vertices. Also, observe that either $H_2$ is a block graph with $J_{H_2}$ Cohen-Macaulay or $H_2$ is a graph satisfying \Cref{setup}. Suppose $b(H_2)=b(H)$ for some $H\in \mathbf{CMB}(H_2)$ and $H=H_2\setminus W$ for some set of cut vertices of $H_2$. If $W(G)$ is the set of all cut vertices of $G$, then $W(G)\setminus \{w\}$ is the set of all cut vertices of $H_2$. Now, consider $H'=G\setminus (W\cup \{w\})$. Then $H'$ is a gluing of $H_1$ and $H$ at the vertex $u$, which implies $H'$ is a block graph with $J_{H'}$ Cohen-Macaulay by \Cref{thmcmblock}. Therefore, we have $H'\in\mathbf{CMB}(G)$ and $b(G)\geq b(H')= b(H_1)+b(H_2)$. By induction hypothesis and the formula for regularity of gluing, we get 
\begin{equation}\label{eq3}
\mathrm{reg}(S/\langle J_{G \setminus \{w\}}, x_w,y_w \rangle) =b(H_1)+b(H_2)\leq b(G).
\end{equation}
Now, we use the exact sequence \eqref{eqexact} by considering the graph $G$ and the non-free vertex $w$ of $G$ as follows:
$$0 \longrightarrow \frac{S}{J_G} \longrightarrow \frac{S}{\langle J_{G \setminus \{w\}},x_w,y_w\rangle} \oplus \frac{S}{J_{G_w}} \longrightarrow \frac{S}{\langle J_{G_w \setminus \{w\}},x_w,y_w\rangle} \longrightarrow 0.$$
Thus, by regularity lemma, we have 
\begin{align}\label{eq4}
   \mathrm{reg}(S/J_G) \leq \mathrm{max}\{\mathrm{reg}(S/\langle J_{G \setminus \{w\}},x_w,y_w\rangle), \mathrm{reg}(S/J_{G_w}), \mathrm{reg}(S/\langle J_{G_w \setminus \{w\}},x_w,y_w\rangle) + 1\}. 
\end{align}
Using inequalities \eqref{eq1}, \eqref{eq2}, \eqref{eq3} and \eqref{eq4}, we have
\begin{align*}
    \reg(S/J_G)&\leq \max\{b(G), b(G)-1, b(G)-1+1\}\\
    &=b(G).
\end{align*}
    Let $b(G)=b(H)$ for some $H\in\mathbf{CMB}(G)$. Then $\reg(S/J_H)=b(H)$ by \Cref{remregblock}. Since $H$ is an induced subgraph of $G$, by \cite[Corollary 2.2]{mm13}, we have $\reg(S/J_G)\geq \reg(S/J_H)=b(H)=b(G)$. Therefore, we have $\reg(S/J_G)=b(G)$.
\end{proof}

\begin{remark}{\rm
    Let $G$ be a graph and $H$ be any induced block graph of $G$ with $J_H$ Cohen-Macaulay. Such a $H$ always exists as removing a sufficient number of vertices from a non-empty graph, we always have a $K_2$. Therefore, by \cite[Corollary 2.2]{mm13} and \Cref{remregblock}, we have $b(H)\leq\reg(S/J_G)$. Hence, we can define $b(G)$ for any graph $G$ as follows:
    $$b(G):=\max\{b(H)\mid H\text{ is an induced block graph of }G \text{ with }J_{H}\text{ Cohen-Macaulay}\}.$$ 
    Then $b(G)\leq \reg(S/J_G)$. Observe that whenever $G$ is the graph satisfying \Cref{setup}, from our result in \Cref{thmchain}, it follows that the definition of $b(G)$ given before \Cref{remark1} and the definition of $b(G)$ given here coincide.
    
}
\end{remark}

\section{Regularity of $r$-regular $r$-connected blocks with whiskers}\label{secreg-connected}

In this section, we explicitly find the regularity of the Cohen-Macaulay binomial edge ideal of a $r$-regular $r$-connected block with whiskers. In general, we find the regularity formula for the Cohen-Macaulay binomial edge ideals given in \cite{sswhisker}.
\par 

Let $K_m$ and $K_n$ be the complete graphs on $m$ and $n$ vertices, respectively, with $m,n\geq 2$. Suppose that $V(K_m) = \{1,\ldots,m\}$ and $V(K_n) = \{m+1, \ldots, m+n\}$. Let $2\leq r \leq \mathrm{min}\{m,n\}$. Define the graph $K_m \star_r K_n$ as follows:
\begin{enumerate}
    \item[$\bullet$] $V(K_m \star_r K_n) = V(K_m) \cup V(K_n)$;
    \item[$\bullet$] $E(K_m \star_r K_n) = E(K_m) \cup E(K_n) \cup \{\{i,m+i\} \mid 1 \leq i \leq r\}$.
\end{enumerate}

It has been proved in \cite{sswhisker} that $\ov{K_m \star_r K_n}$ (\Cref{fig5}) is Cohen-Macaulay if and only if 
\begin{enumerate}
\item[$\bullet$] $V\ov{(K_m \star_r K_n)} = V(K_m \star_r K_n) \cup \{m+n+1, \ldots, m+n+2(r-1)\}$;
\item[$\bullet$] $E\ov{(K_m \star_r K_n)} = E(K_m \star_r K_n) \cup \{\{i,m+n+i\} \mid 1 \leq i \leq r-1 \} \cup \{\{m+i,m+n+r-1+i\} \mid 1 \leq i \leq r-1\}$.
\end{enumerate}

\begin{figure}[h]
\begin{tikzpicture}
 [scale=1]
 
\filldraw[black] (0.5,0) circle (2pt)node[anchor=east]{$9$};
\filldraw[black] (2,0) circle (2pt)node[anchor=north]{$2$};
\filldraw[black] (2,2) circle (2pt)node[anchor=south]{$1$};
\filldraw[black] (1,1) circle (2pt)node[anchor=east]{$4$};
\filldraw[black] (3,1) circle (2pt)node[anchor=north]{$3$};
\filldraw[black] (0.5,2) circle (2pt)node[anchor=east]{$8$};

\filldraw[black] (5,1) circle (2pt)node[anchor=north]{$7$};
\filldraw[black] (6,2) circle (2pt)node[anchor=south]{$5$};
\filldraw[black] (6,0) circle (2pt)node[anchor=north]{$6$};
\filldraw[black] (7.5,0) circle (2pt)node[anchor=west]{$11$};
\filldraw[black] (7.5,2) circle (2pt)node[anchor=west]{$10$};

\draw[black] (0.5,0) -- (2,0);
\draw[black] (2,0) -- (2,2);
\draw[black] (2,0) -- (1,1);
\draw[black] (2,0) -- (3,1);
\draw[black] (2,2) -- (1,1);
\draw[black] (2,2) -- (3,1);
\draw[black] (1,1) -- (3,1);
\draw[black] (0.5,2) -- (2,2);
\draw[black] (3,1) -- (5,1);
\draw[black] (5,1) -- (6,0);
\draw[black] (5,1) -- (6,2);
\draw[black] (6,0) -- (6,2);
\draw[black] (6,0) -- (2,0);
\draw[black] (6,2) -- (2,2);
\draw[black] (6,0) -- (7.5,0);
\draw[black] (6,2) -- (7.5,2);

\end{tikzpicture}
\caption{The graph $\overline{K_{4}\star_{3} K_{3}}$.}
\label{fig5}
\end{figure}
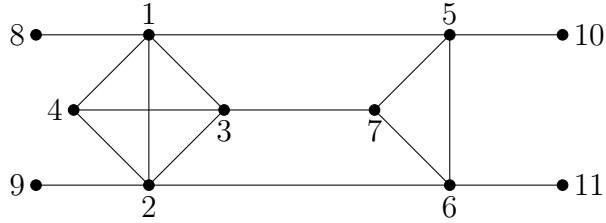

\begin{theorem}\label{thmstar2}
Let $G = \ov{K_m \star_2 K_n}$. Then we have the following:
\begin{align*}
    \reg(S/J_G)=
    \begin{cases}
        4 & \text{if $m>2$ or $n>2$}\\
        3 & \text{otherwise}.
    \end{cases}
\end{align*}
\end{theorem}

\begin{proof}
    First, let us assume $m>2$. Note that $1$ and $m+1$ are the only cut vertices of $G$. Since $m>2$, it is easy to see that $G \setminus \{1\}$ is a block graph consisting of $4$ (non-trivial) blocks with an isolated vertex $m+n+1$ such that $J_{G\setminus \{1\}}$ is Cohen-Macaulay. Therefore, $\reg(S/J_{G\setminus\{1\}})=4$ by \Cref{remregblock}. Now, we consider the graph $G_1$. Then $G_1=\mathrm{cone}(m+1,H)$, where $H=H_1\sqcup H_2$, $H_1$ is a block graph with $J_{H_1}$ Cohen-Macaulay and $H_2$ is the isolated vertex $m+n+2$. Note that $H_1$ consists of $2$ blocks if $n=2$, otherwise $H_1$ has $3$ blocks. Thus, using \cite[Theorem 2.1]{mkreg3} and \Cref{remregblock}, we have $\reg(S/J_{G_1})\leq 3$. Since $1$ is a free vertex in $G_1$ and the clique of $G_1$ containing $1$ has at least $5$ vertices, we can use similar arguments as $G_1$ for $G_{1}\setminus\{1\}$ to get $\reg(S/J_{G_{1}\setminus \{1\}})\leq 3$. The short exact sequence \eqref{eqexact} corresponding to the graph $G$ and the non-free vertex $1$ of $G$ become
    $$0 \longrightarrow \frac{S}{J_G} \longrightarrow \frac{S}{\langle J_{G \setminus \{1\}},x_1,y_1\rangle} \oplus \frac{S}{J_{G_1}} \longrightarrow \frac{S}{\langle J_{G_1 \setminus \{1\}},x_1,y_1\rangle} \longrightarrow 0.$$
Thus, by regularity lemma, we have 
\begin{align*}
    \mathrm{reg}(S/J_G) &\leq \max\{\mathrm{reg}(S/\langle J_{G \setminus \{1\}},x_1,y_1\rangle), \mathrm{reg}(S/J_{G_1}), \mathrm{reg}(S/\langle J_{G_1 \setminus \{1\}},x_1,y_1\rangle) + 1\}\\
    &\leq \max\{4,3,3+1\}=4.
\end{align*}
  Since $G\setminus \{1\}$ is an induced subgrph of $G$, by \cite[Corollary 2.2]{mm13}, it follows that $\reg(S/J_G)\geq\reg(S/J_{G\setminus \{1\}})=4$. Hence, $\reg(S/J_G)=4$. The case of $n>2$ is similar by considering the cut vertex $m+1$ instead of $1$.\par 

  Now, we suppose $m=2$ and $n=2$. Then $G$ satisfies \Cref{setup} and it is easy to verify $b(G)=b(G\setminus\{1\})=b(G\setminus\{3\})=3$. Therefore, by \Cref{thmchain}, we get $\reg(S/J_{G})=3$.
\end{proof}

\begin{figure}[h]
\begin{tikzpicture}
 [scale=1]
 
\filldraw[black] (0.5,0) circle (2pt)node[anchor=east]{$9$};
%\filldraw[black] (2,0) circle (2pt)node[anchor=north]{$2$};
\filldraw[black] (2,2) circle (2pt)node[anchor=south]{$1$};
\filldraw[black] (1,1) circle (2pt)node[anchor=east]{$4$};
\filldraw[black] (3,1) circle (2pt)node[anchor=north]{$3$};
\filldraw[black] (0.5,2) circle (2pt)node[anchor=east]{$8$};

\filldraw[black] (5,1) circle (2pt)node[anchor=north]{$7$};
\filldraw[black] (6,2) circle (2pt)node[anchor=south]{$5$};
\filldraw[black] (6,0) circle (2pt)node[anchor=north]{$6$};
\filldraw[black] (7.5,0) circle (2pt)node[anchor=west]{$11$};
\filldraw[black] (7.5,2) circle (2pt)node[anchor=west]{$10$};

%\draw[black] (0.5,0) -- (2,0);
%\draw[black] (2,0) -- (2,2);
%\draw[black] (2,0) -- (1,1);
%\draw[black] (2,0) -- (3,1);
\draw[black] (2,2) -- (1,1);
\draw[black] (2,2) -- (3,1);
\draw[black] (1,1) -- (3,1);
\draw[black] (0.5,2) -- (2,2);
\draw[black] (3,1) -- (5,1);
\draw[black] (5,1) -- (6,0);
\draw[black] (5,1) -- (6,2);
\draw[black] (6,0) -- (6,2);
%\draw[black] (6,0) -- (2,0);
\draw[black] (6,2) -- (2,2);
\draw[black] (6,0) -- (7.5,0);
\draw[black] (6,2) -- (7.5,2);

\end{tikzpicture}
\caption{The graph $G\setminus\{2\}$, where $G=\overline{K_{4}\star_{3} K_{3}}$.}
\label{fig6}
\end{figure}
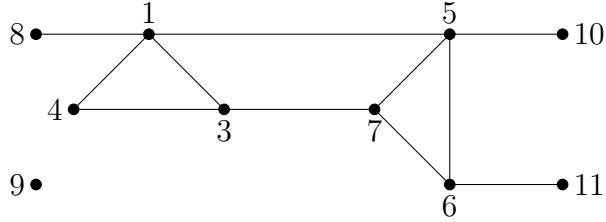

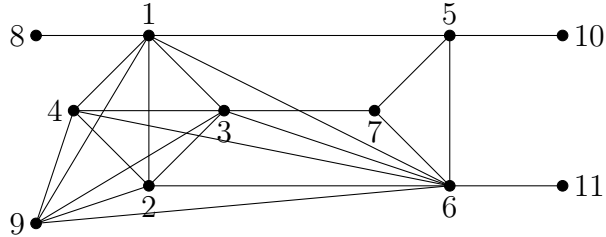
\begin{figure}[h]
\begin{tikzpicture}
 [scale=1]
 
\filldraw[black] (0.5,-0.5) circle (2pt)node[anchor=east]{$9$};
\filldraw[black] (2,0) circle (2pt)node[anchor=north]{$2$};
\filldraw[black] (2,2) circle (2pt)node[anchor=south]{$1$};
\filldraw[black] (1,1) circle (2pt)node[anchor=east]{$4$};
\filldraw[black] (3,1) circle (2pt)node[anchor=north]{$3$};
\filldraw[black] (0.5,2) circle (2pt)node[anchor=east]{$8$};

\filldraw[black] (5,1) circle (2pt)node[anchor=north]{$7$};
\filldraw[black] (6,2) circle (2pt)node[anchor=south]{$5$};
\filldraw[black] (6,0) circle (2pt)node[anchor=north]{$6$};
\filldraw[black] (7.5,0) circle (2pt)node[anchor=west]{$11$};
\filldraw[black] (7.5,2) circle (2pt)node[anchor=west]{$10$};

\draw[black] (0.5,-0.5) -- (2,0);
\draw[black] (2,0) -- (2,2);
\draw[black] (2,0) -- (1,1);
\draw[black] (2,0) -- (3,1);
\draw[black] (2,2) -- (1,1);
\draw[black] (2,2) -- (3,1);
\draw[black] (1,1) -- (3,1);
\draw[black] (0.5,2) -- (2,2);
\draw[black] (3,1) -- (5,1);
\draw[black] (5,1) -- (6,0);
\draw[black] (5,1) -- (6,2);
\draw[black] (6,0) -- (6,2);
\draw[black] (6,0) -- (2,0);
\draw[black] (6,2) -- (2,2);
\draw[black] (6,0) -- (7.5,0);
\draw[black] (6,2) -- (7.5,2);

\draw[black] (0.5,-0.5) -- (6,0);
\draw[black] (0.5,-0.5) -- (3,1);
\draw[black] (0.5,-0.5) -- (2,2);
\draw[black] (0.5,-0.5) -- (1,1);
\draw[black] (6,0) -- (3,1);
\draw[black] (6,0) -- (1,1);
\draw[black] (6,0) -- (2,2);

\end{tikzpicture}
\caption{The graph $G_v$, where $G=\overline{K_{4}\star_{3} K_{3}}$ and $v=2$.}
\label{fig7}
\end{figure}

\begin{lemma}\label{lemstarupper}
Let $G = \ov{K_m \star_r K_n}$. Then $\mathrm{reg}(S/J_G)\leq 2r-1$ for $r\geq 3$.
\end{lemma}
\begin{proof}
We first prove the result for $r=3$, and then we proceed by induction on $r$. Note that $1,2,m+1,m+2$ are the only cut vertices of $G=\ov{K_m\star_3 K_n}$. We will consider the cut vertex $2$ and investigate the graphs $G\setminus\{2\}$, $G_2$ and $G_{2}\setminus \{2\}$. It is easy to see from the structure of $G$ that $G\setminus \{2\}$ is a gluing of $H_1$ and $H_2$ at the vertex $m+2$, where $H_1\cong \ov{K_{m-1}\star_{2} K_{n}}$ and $H_{2}\cong K_2$ with $V(H_2)=\{m+2,m+n+4\}$ and also, $G\setminus\{2\}$ has an isolated vertex $m+n+2$. Therefore, by \Cref{thmstar2} and \Cref{remregblock}, we have
\begin{align}\label{eq4.1}
\mathrm{reg}(S/\langle J_{G \setminus \{2\}}, x_{2},y_{2} \rangle) = \mathrm{reg}(S_1/J_{H_1}) + \mathrm{reg}(S_2/J_{H_2}) \leq 4+1 = 5,
\end{align}
where $S_1$ and $S_2$ are the corresponding polynomial rings of $J_{H_1}$ and $J_{H_2}$ respectively. Now, let us focus on the graph $G_2$. We choose the cut vertex $m+2$ of $G_2$. Then $G_2 \setminus \{m+2\} \cong \ov{K_{m+1} \star_{2} K_{n-1}} \sqcup \{m+n+4\}$, and thus, $\mathrm{reg}(S/\langle J_{G_2 \setminus \{m+2\}}, x_{m+2},y_{m+2} \rangle) \leq 4$ by \Cref{thmstar2}. Also, note that $(G_2)_{m+2}\cong \ov{K_{m+n+2}}$, where $\ov{K_{m+n+2}}$ is a complete graph with $2$ whiskers attached to $2$ distinct vertices. Therefore, by \Cref{remregblock}, we get $\reg(S/J_{(G_2)_{m+2}})=3$. Again, $m+2$ is a free vertex in $(G_2)_{m+2}$, and the block containing $m+2$ has more than two vertices. Hence, $(G_2)_{m+2} \setminus \{m+2\}$ is a Cohen-Macaulay block graph with $3$ blocks. Therefore, by \Cref{remregblock},
$$\mathrm{reg}(S/\langle J_{(G_2)_{m+2} \setminus \{m+2\}}, x_{m+2},y_{m+2} \rangle)=3.$$
By considering the exact sequence \eqref{eqexact} with the graph $G_2$ and the non-free vertex $v=m+2$ of $G_2$ and using the regularity lemma, we have
\begin{align*}
    \mathrm{reg}(S/J_{G_2}) &\leq \mathrm{max}\{\mathrm{reg}(S/\langle J_{G_2 \setminus \{m+2\}},x_{m+2},y_{m+2}\rangle), \mathrm{reg}(S/J_{(G_2)_{m+2}}), \\
    &\hspace{1.5cm} \mathrm{reg}(S/\langle J_{(G_2)_{m+2} \setminus {m+2}},x_{m+2},y_{m+2}\rangle) + 1\}.
\end{align*}
In particular, 
\begin{align}\label{eq4.2}
    \reg(S/J_{G_2})\leq \mathrm{max}\{4,3,3+1\}=4.
\end{align}
 \noindent Now, we consider the graph $G_2 \setminus \{2\}$. Note that $2$ and $m+n+2$ are free vertices in $G_2$, and the maximal clique containing $2$ in $G_2$ also contains $m+n+2$. Thus, the structure of $G_2\setminus \{2\}$ and $G_2$ are almost same. Hence, performing the same operations and arguments to $G_2\setminus \{2\}$ as we have done to get the upper bound of $\reg(S/J_{G_2})$, one can easily derive 
\begin{align}\label{eq4.3}
    \mathrm{reg}(S/ \langle J_{G_2 \setminus \{2\}}, x_2, y_2 \rangle) \leq 4.
\end{align}
Now, consider the exact sequence \eqref{eqexact} by taking the graph $G$ and the vertex $v=2$ as follows:
$$0 \longrightarrow \frac{S}{J_G} \longrightarrow \frac{S}{\langle J_{G \setminus \{2\}},x_2,y_2\rangle} \oplus \frac{S}{J_{G_2}} \longrightarrow \frac{S}{\langle J_{G_2 \setminus \{2\}},x_2,y_2\rangle} \longrightarrow 0.$$
Thus, applying the regularity lemma and using inequalities \eqref{eq4.1}, \eqref{eq4.2} and \eqref{eq4.3}, we get 
\begin{align*}
    \mathrm{reg}(S/J_G) &\leq \mathrm{max}\{\mathrm{reg}(S/\langle J_{G \setminus \{2\}},x_2,y_2\rangle), \mathrm{reg}(S/J_{G_2}), \mathrm{reg}(S/\langle J_{G_2 \setminus \{2\}},x_2,y_2\rangle) + 1\}\\
    &= \mathrm{max}\{5,4,4+1\}=5=2\times 3-1.
\end{align*}
Next, we consider $r>3$ and assume the required inequality holds for the graph $\ov{K_{m}\star_{r-1} K_n}$. Note that $r-1$ is a cut vertex of $G=\ov{K_{m}\star_{r} K_n}$. For simplicity of notation, let us denote $r-1$ by $v$. Then, one can see that $G \setminus \{v\}$ (see \Cref{fig6}) is a gluing of $H_1$ and $H_2$ at the vertex $m+r-1$, where $H_1 \cong \ov{K_{m-1} \star_{r-1} K_n}$ and $H_2 \cong K_2$ with $V(H_2) = \{m+r-1, m+n+2(r-1)\}$. Also, $G \setminus \{v\}$ has an isolated vertex $m+n+r-1$. Therefore, by induction hypothesis and \Cref{remregblock}, we have
\begin{align}\label{eq4.4}
    \mathrm{reg}(S/\langle J_{G \setminus \{v\}}, x_{v},y_{v} \rangle) = \mathrm{reg}(S_1/J_{H_1}) + \mathrm{reg}(S_2/J_{H_2}) \leq 2(r-1)-1+1 = 2r-2,
\end{align}
where $S_1$ and $S_2$ are the corresponding polynomial rings of $J_{H_1}$ and $J_{H_2}$ respectively.
Now, we consider the graph $G_v$ (see \Cref{fig7}). We claim that $\mathrm{reg}(S/J_{G_v}) \leq 2r-2$. Let $H = G_v$, and we choose the cut vertex $m+r-1$ of $H$. For simplicity, let us denote $m+r-1$ by $u$. Observe that $H \setminus \{u\} \cong \ov{K_{m+1} \star_{r-1} K_{n-1}} \sqcup \{m+n+2(r-1)\}$. Thus, by induction hypothesis, we get 
$$\mathrm{reg}(S/\langle J_{H \setminus \{u\}}, x_u,y_u \rangle)\leq 2(r-1)-1 = 2r-3.$$
Now, consider the graph $H_u$. From the structure of $G$, it is easy to verify that $H_u \cong (G_v)_u \cong \ov{K_{m+n+2}},$ where $\ov{K_{m+n+2}}$ is a complete graph with $2(r-2)$ whiskers attached with $2(r-2)$ distinct vertices. Therefore, $H_u$ is a block graph with $2(r-2)+1$ blocks and $J_{H_u}$ is Cohen-Macaulay. Thus, it follows from \Cref{remregblock} that $\mathrm{reg}(S/J_{H_u}) = 2(r-2)+1 = 2r-3$. Again, $u$ is a free vertex in $H_u$, and the maximal clique containing $u$ in $H_u$ has more than $m+n+2$ vertices. Hence, $H_u \setminus \{u\}$ is again a block graph with $2(r-2) + 1$ blocks and $J_{H_u\setminus \{u\}}$ is Cohen-Macaulay. Similarly, we have $\mathrm{reg}(S/\langle J_{H_u \setminus \{u\}}, x_u,y_u \rangle)= 2r-3$. Now, considering the exact sequence \eqref{eqexact} with the graph $H$ and the vertex $u$, the regularity lemma gives the following:
$$\mathrm{reg}(S/J_H) \leq \mathrm{max}\{\mathrm{reg}(S/\langle J_{H \setminus u},x_u,y_u\rangle), \mathrm{reg}(S/J_{H_u}), \mathrm{reg}(S/\langle J_{H_u \setminus u},x_u,y_u\rangle) + 1\}.$$
Thus, using previous inequalities, we get
\begin{align}\label{eq4.5}
    \mathrm{reg}(S/J_{G_v}) \leq \mathrm{max}\{2r-3,2r-3,2r-2\} = 2r-2.
\end{align}
This proves the claim. Since $v$ and $m+n+r-1$ are free vertices in $G_v$ belonging to the same maximal clique of $G_v$, the structure of $G_v$ and $G_v \setminus \{v\}$ are almost similar. Thus, performing the same operations and arguments to $G_v\setminus \{v\}$ as we have given in the proof of the previous claim for $G_v$, we can easily obtain 
\begin{align}\label{eq4.6}
    \mathrm{reg}(S/ \langle J_{G_v \setminus \{v\}}, x_v, y_v \rangle) \leq 2r-2.
\end{align}
Now, let us consider the following exact sequence:
$$0 \longrightarrow \frac{S}{J_G} \longrightarrow \frac{S}{\langle J_{G \setminus \{v\}},x_v,y_v\rangle} \oplus \frac{S}{J_{G_v}} \longrightarrow \frac{S}{\langle J_{G_v \setminus \{v\}},x_v,y_v\rangle} \longrightarrow 0.$$
Therefore, by inequalities \eqref{eq4.4}, \eqref{eq4.5}, \eqref{eq4.6} and the regularity lemma, we have 
\begin{align*}
    \mathrm{reg}(S/J_G) &\leq \mathrm{max}\{\mathrm{reg}(S/\langle J_{G \setminus \{v\}},x_v,y_v\rangle), \mathrm{reg}(S/J_{G_v}), \mathrm{reg}(S/\langle J_{G_v \setminus \{v\}},x_v,y_v\rangle) + 1\}\\
    &\leq \mathrm{max}\{2r-2,2r-2,2r-1\}=2r-1.
\end{align*}
This completes the proof.
\end{proof}

To establish the lower bound of $\reg(S/J_{\ov{K_m\star_{r} K_n}})$, we need some theory of square-free monomial ideals and initial ideals. Let us first recall them first.
\par 

Let $R=\mathbb{K}[x_{1},\ldots, x_{n}]$ be a polynomial ring over a field $\mathbb{K}$. An ideal $I\subset R$ is said to be a \textit{monomial ideal} if $I$ is generated by a set of monomials. For a monomial ideal $I$, the minimal monomial generating set is unique, which is denoted by $G(I)$. A monomial ideal $I$ is said to be \textit{square-free} if $G(I)$ consists of only square-free monomials. Let $J\subset R$ be a graded ideal and $<$ be a monomial term order on $R$. For a polynomial $f\in R$, we write $\ini_{<}(f)$ to denote the term with the largest monomial in $f$ with respect to $<$. The \textit{initial ideal} of $J$ with respect to $<$, denoted by $\ini_{<}(J)$, is the monomial ideal of $R$ generated by $\{\ini_{<}(f)\mid f\in J\}$.

\begin{definition}{\rm
A \textit{simple hypergraph} $\mathcal{H}$ is a pair $(V(\mathcal{H}),E(\mathcal{H}))$, where $V(\mathcal{H})$ is a set of finite elements, known as the \textit{vertex set} of $\mathcal{H}$ and $E(\mathcal{H})$ is a collection of subsets of $V(\mathcal{H})$ such that no two elements of $E(\mathcal{H})$ contain each other, called the \textit{edge set} of $\mathcal{H}$. Elements of $V(\mathcal{H})$ are called vertices of $\mathcal{H}$ and elements of $E(\mathcal{H})$ are called edges of $\mathcal{H}$.
}
\end{definition}

Let $\mathcal{H}$ be a simple hypergraph on the vertex set $V(\mathcal{H})=\{x_{1},\ldots,x_{n}\}$. For $A\subset V(\mathcal{C})$, we consider $X_{A}:=\prod_{x_{i}\in A} x_{i}$ as a square-free monomial in the polynomial ring $R$. The \textit{edge ideal} of the hypergraph $\mathcal{H}$, denoted by $I(\mathcal{H})$, is an ideal of $R$ defined by 
$$I(\mathcal{H})=\big<X_{e}\mid e\in E(\mathcal{H})\big>.$$
In this sense, the family of square-free monomial ideals corresponds one-to-one with the family of simple hypergraphs. 
%For a square-free monomial ideal $I$ of $R$, we denote by $\mathcal{H}(I)$ the Corresponding simple hypergraph.

\begin{definition}{\rm
An \textit{induced matching} in a simple hypergraph $\mathcal{H}$ is a set of pairwise disjoint edges $e_{1},\ldots,e_{r}$ such that the only edges of $\mathcal{H}$ contained in $\bigcup_{i=1}^{r} e_{i}$ are $e_{1},\ldots,e_{r}$.
}
\end{definition}

\begin{proposition}[{\cite[Corollary 3.9]{mv12}}]\label{propim}
Let $\mathcal{H}$ be a simple graph and $M=\{e_{1},\ldots, e_{r}\}$ be an induced matching in $\mathcal{H}$. Then $\sum_{i=1}^{r}(\vert e_{i}\vert -1)=(\sum_{i=1}^{r}\vert e_{i}\vert)-r\leq \mathrm{reg}(R/I(\mathcal{H}))$.
\end{proposition}

\begin{definition}[{\cite{hhhrkara}}]{\rm
Let $G$ be a graph with $V(G)=[n]$. A path $\pi: i=i_{0},i_{1},\ldots,i_{r}=j$ from $i$ to $j$ with $i<j$ in $G$ is said to be an \textit{admissible path} in $G$ if the following hold:
\begin{enumerate}
\item[(i)] $i_{k}\neq i_{l}$ for $k\neq l$;

\item[(ii)] For each $k\in\{1,\ldots,r-1\}$, either $i_{k}<i$ or $i_{k}>j$;

\item[(iii)] The induced subgraph of $G$ on the vertex set $\{i_{0},\ldots,i_{r}\}$ has no induced cycle.
\end{enumerate}
}
\end{definition}

\begin{remark}\label{order}{\rm
Corresponding to an admissible path $\pi: i=i_{0},i_{1},\ldots,i_{r}=j$ from $i$ to $j$ with $i<j$ in $G$, we associate the monomial
$$ u_{\pi}=\bigg(\prod_{i_{k}>j} x_{i_{k}}\bigg)\bigg(\prod_{i_{l}<i} y_{i_{l}}\bigg).$$
Let $\prec$ be the lexicographic order on $S=\mathbb{K}[x_1,\ldots,x_n,y_1,\ldots,y_n]$ induced by $x_1\succ x_2\succ \cdots \succ x_n \succ y_1\succ y_2\succ \cdots\succ y_n$. Then $\mathcal{G}=\{u_{\pi}f_{ij}\mid \pi\,\, \text{is an admissible path from}\,\, i\,\, \text{to}\,\, j\,\, \text{with}\,\, i<j\}$ is a reduced Gr\"{o}bner basis of $J_{G}$ with respect to $\prec$ by \cite[Theorem 2.1]{hhhrkara}. Therefore, 
$$G(\mathrm{in}_{\prec}(J_{G}))=\{u_{\pi}x_{i}y_{j}\mid \pi\,\, \text{is an admissible path from}\,\, i\,\, \text{to}\,\, j\,\, \text{with}\,\, i<j\}.$$
}
\end{remark}

To get the lower bound of $\mathrm{reg}(S/J_G)$, where $G = \ov{K_m \star_r K_n} $, we use the lower bound of the regularity of the initial ideal of $J_G$ with respect to the term order mentioned in \Cref{order}. Since the labelling on the vertices does not change the regularity of $J_G$, we use the following labelling for our purpose. Note that $m,n \geq r$. Consider
\begin{enumerate}
\item[$\bullet$] $V(K_m) = \{1,3, \ldots, 2r-1, 2r+1, 2r+2, \ldots, 2r+(m-r)\};$
\item[$\bullet$] $V(K_n) = \{2,4, \ldots, 2r, 2r+(m-r)+1, 2r+(m-r)+2, \ldots, m+n\}.$
\end{enumerate}
Thus, the edge sets of $K_m \star_r K_n$ and $\ov{K_m \star_r K_n}$ will be the following.
\begin{enumerate}
\item[$\bullet$] $E(K_m \star_r K_n) = E(K_m) \cup E(K_n) \cup \{\{2i-1,2i\} \mid 1 \leq i \leq r \};$
\item[$\bullet$] $E(\ov{K_m \star_r K_n}) = E(K_m \star_r K_n) \cup \{\{i+2,m+n+i\} \mid 1 \leq i \leq 2r-2\}.$
\end{enumerate}

\begin{lemma}\label{lemstarlower}
Let $G = \ov{K_m \star_r K_n}$. Then $\mathrm{reg}(S/J_G) \geq 2r-1$.
\end{lemma}
\begin{proof}
Note that $r \geq 2$ by definition of $\ov{K_m \star_r K_n}$. First, we will investigate $\reg(S/\ini_{\prec}(J_G))$, where $\ini_{\prec}(J_G)$ denotes the initial ideal of $J_G$ with respect to the term order $\prec$ defined in \Cref{order}.Let $\mathcal{H}$ denote the corresponding hypergraph of $\ini_{\prec}(J_G)$, i.e., $I(\mathcal{H}) = \ini_{\prec}(J_G)$. Take the following set $\mathcal{M}$ of the edges of the hypergraph $\mathcal{H}$:
$$\mathcal{M} = \{\{x_1,y_2\}, \{x_{i+2},y_{m+n+i}\} \mid 1 \leq i \leq 2r-2\}.$$ 
We claim that $\mathcal{M}$ forms an induced matching in $\mathcal{H}$. It is clear that $\mathcal{E}_1 \cap \mathcal{E}_2 = \emptyset$ for any two distinct edges $\mathcal{E}_1, \mathcal{E}_2 \in \mathcal{M}$. Let $L = \cup_{\mathcal{E}\in\mathcal{M}} \mathcal{E}= \{x_1,y_2\} \cup \{x_{i+2}, y_{m+n+i} \mid 1 \leq i \leq 2r-2\}$. Suppose there exists $\mathcal{E} \in E(\mathcal{H})$ such that $\mathcal{E} \notin \mathcal{M}$ but $\mathcal{E} \subset L$. Note that $\{x_{i+2},y_2\} \notin E(\mathcal{H})$ as $i+2 > 2$. Also, $\{x_i,y_{m+n+j}\} \notin E(\mathcal{H})$ if $i \neq j+2$ as the only neighbour of $m+n+j$ in $G$ is $j+2$. Therefore, $\vert \mathcal{E} \vert > 2.$ Now, corresponding to $\mathcal{E}$, there exists an admissible path $\pi: i = i_0, i_1, \ldots, i_k = j$ in $G$ such that $X_{\mathcal{E}}=u_{\pi} x_iy_j$. Therefore, $i \in \{1, 3,4,5 \ldots, 2r\} = A$ and $j \in \{2,m+n+1, m+n+2, \dots, m+n+(2r-2)\} = B$. Now, $i_1 < i$ or $i_1 > j$ as $\pi$ is an admissible path in $G$. Also, we have $i < j$. Now, we consider two cases: \par 
\noindent \textbf{Case-1:} Suppose $i_1 < i$. Then $y_{i_1}$ divides $u_{\pi} x_iy_j$. This implies $y_{i_1} \in \mathcal{E}$. Since $\mathcal{E} \subset L$, $y_{i_1} \in L$. Also, we have $y_j \in \mathcal{E} \subset L$. Since $i_1 < i < j$ and $y_{i_1},y_j \in L$, the only choice of $i_1$ is $2$ as the other vertices of $B$ are free vertices. Clearly, $\{i_1,j\} \notin E(G)$ and so, $i_2 \neq j$. Also, $i_2 \neq 1$ as $\mathcal{E} \subset L$. Therefore, $i_2$ and $i$ are two neighbours of $i_1 = 2$ other than $1$. Thus, $\{i_2, i\} \in E(G)$ by the labelling of $G$, and this gives a contradiction as $\pi$ is an admissible path. \par 
\noindent \textbf{Case-2:} Suppose $i_1 > j$. Then $x_{i_1}$ divides $u_{\pi} x_iy_j$. This implies $x_{i_1} \in \mathcal{E} \subset L$, and thus, $i_1 \in A$. Since $i < j < i_1$ and $x_i,x_{i_1} \in L$, we get $i, i_1 \in A$. If $j=2$, then $i=1$, which is not possible as $\vert \mathcal{E}\vert>2$ and $\{1,2\}\in E(G)$. Therefore, $j \in B\setminus \{2\}$. Note that there is no element in $A$, which can be greater than any element of $B \setminus \{2\}$. Hence, $i_1 \ngtr j$.\par

 Both cases show that such an admissible path $\pi$ does not exist. Therefore, $\mathcal{M}$ forms an induced matching in $\mathcal{H}$. Therefore, by \Cref{propim}, we have
 $$\mathrm{reg}(S/ \ini_{\prec}(J_G)) \geq \left(\sum_{\mathcal{E} \in \mathcal{M}} \vert \mathcal{E} \vert - 1\right) = 2(r-1)+1 = 2r-1.$$
Since $\ini_{\prec}(J_G)$ is a square-free monomial ideal, $\reg(S/J_G)=\reg(S/\ini_{\prec}(J_G))$ by \cite[Corollary 2.7]{cv20}. Hence, $\reg(S/J_G)\geq 2r-1$.
\end{proof}

\begin{theorem}\label{thmregcon}
    Let $G = \ov{K_m \star_r K_n}$ and $r\geq 3$. Then $\mathrm{reg}(S/J_G) =2r-1$.
\end{theorem}
\begin{proof}
    The result follows from \Cref{lemstarupper} and \Cref{lemstarlower}.
\end{proof}

\noindent \textbf{Acknowledgements:} Kamalesh Saha wants to thank the National Board for Higher Mathematics (India) for the financial support through the NBHM Postdoctoral Fellowship and Chennai Mathematical Institute for providing a good research environment. An Infosys Foundation fellowship partially supports Kamalesh Saha.

\printbibliography

\end{document}